\documentclass{article}
\usepackage[utf8]{inputenc}
\usepackage{latexsym}
\usepackage{amsmath, amsfonts, amssymb, amscd, amsthm, url}
\usepackage{graphicx}
\usepackage{mathtools}
\usepackage[export]{adjustbox}
\usepackage[english]{babel}

\usepackage[left=2cm,right=2cm,
    top=2cm,bottom=2cm,bindingoffset=0cm]{geometry}
    
\usepackage[
    draft = false,
    unicode = true,
    colorlinks = true,
    allcolors = blue,
    hyperfootnotes = true,
    citecolor = red
]{hyperref}

\DeclareMathOperator{\conv}{Conv}

\newcommand{\R}{\ensuremath{\mathbb{R}}}

\theoremstyle{plain}
    \newtheorem{theorem}{Theorem}
    \newtheorem{assertion}[theorem]{Assertion}
    \newtheorem{lemma}[theorem]{Lemma}
    
    \newtheorem{remark}[theorem]{Remark}
    \newtheorem*{remark*}{Remark}

\renewcommand{\leq}{\leqslant}
\renewcommand{\le}{\leqslant}

\title{An example of an “unlinked” set of $2k + 3$ points in $2k$-space\footnote{We are grateful to S. Dzhenzher and A. Skopenkov for helpful comments and to M.Tancer for a draft proof of Lemma~\ref{l:mcsi}.}}
\author{Starkov Mikhail}
\date{}

\begin{document}
\maketitle
\begin{abstract}
Take any $d + 3$ points in $\mathbb{R}^d$. It is known that 

(a) if $d = 2k + 1$, then there are two linked $(k + 1)$-simplices with the vertices at these points; 

(b) if $d = 2k$, then there are two disjoint $(k + 1)$-tuples of these points such that their convex hulls intersect.

The analogue of (b) for $d = 2k + 1$, which is also the analogue of (a) for intersections (instead of linkings), states that there are two disjoint $(k + 1)$- and $(k + 2)$-tuples of these points such that their convex hulls intersect. This analogue is correct by (a).

In this paper we disprove the analogue of (a) for $d = 2k$, which is also the analogue of (b) for linkings (instead of intersections).
\end{abstract}
\bigskip
In this paper we present a result (Assertion~\ref{a:main}) on parity of number of intersections for system of points in $\R^d$. 
Let us describe the motivation at first.

Points in $\R^d$ are in \emph{general position} if no $d + 1$ of them lie in the same $(d - 1)$-hyperplane. A pair of $k$-simplex and $(d - k)$-simplex in $\R^d$ is called \emph{linked} if the $k$-simplex intersects the boundary of the $(d - k)$-simplex at an odd number of points.

The following theorem is well known (see survey \cite[Theorem 1.6]{Sk22}).

\begin{theorem}\label{link1}
Take any $d + 3$ points in general position in $\R^d$.

(a)  If $d = 2k + 1$, then there are two linked $(k + 1)$-simplices with the vertices at these points. 
    
(b)  If $d = 2k$, then there are two disjoint $(k + 1)$-tuples of these points such that their convex hulls intersect.
\end{theorem}

The analogue of (b) for $d = 2k + 1$, which is also the analogue of (a) for intersections (instead of linkings), states that there are two disjoint $(k + 1)$- and $(k + 2)$-tuples of these points such that their convex hulls intersect.
This analogue is correct by (a).

In this paper we disprove the analogue of (a) for $d = 2k$, which is also the analogue of (b) for linkings (instead of intersections).
\footnote{This was formulated without a proof in 2015 \cite[arXiv version 3, Proposition 5.1.a]{Sk22} and explicitly repeated as a conjecture in 2018 \cite[arXiv version 4, Remark~2.9, item~(4-3)]{Sk22} (along with the info on M. Tancer's proof of the PL version).}

\begin{assertion}\label{a:main}
There are $2k + 3$ points in general position in $\R^{2k}$ such that any $k$-simplex formed by $k + 1$ of them intersects the boundary of the simplex formed by the remaining $k + 2$ points at an even number of points.
\end{assertion}
For $k = 1$ one can choose as an example the $5$ vertices of a convex $5$-polygon. However, it is not easy to generalize this construction to higher  dimensions, so we use a different one. 
Recall that $[n] := \{1,2,\ldots,n\}$ for a positive integer $n$.
The \emph{moment curve} in $\R^d$ is a map $m \colon \R \to \R^d$ given by the formula $m(t) = (t, t^2,\ldots, t^d)$.
We use $m([2k + 3])$ as an example to Assertion~\ref{a:main}.
Below we reduce Assertion~\ref{a:main} to the purely combinatorial Lemma~\ref{a:comb}.

\begin{remark}\label{a:extra}
For any $2k + 3$ points in general position in $\R^{2k}$ the number of $k$-simplices formed by $k+1$ of them, that intersect the surface of the $(k + 1)$-simplex  formed by the remaining $k + 2$ points exactly at one point, is even.
\end{remark}
Remark~\ref{a:extra} means that the number of linked pairs of $k$-simplex and $(k + 1)$-simplex is even. 
Assertion~\ref{a:main} claims that this number can be zero.

Lemmas~\ref{corr},~\ref{l:mcl},~\ref{l:mcsi} are well known (see \cite{Br} for Lemma~\ref{l:mcsi}). For the reader's convenience, we present the proofs.

For a set $X$ and a positive integer $s$ denote by ${X \choose s}$ the set of $s$-element subsets of $X$.
Recall that $\conv{X}$ is the convex hull of a finite subset $X \subset \R^d$. In other words, $\conv{X}$ is a simplex formed by the points from $X$.
\begin{lemma}\label{corr}
    For any $d + 3$ points in general position in $\R^d$ any $k$-simplex formed by $k + 1$ of these points intersects the $(d - k)$-simplex formed by $d + 1 - k$ of the remaining $d + 2 - k$ points in at most one point.
\end{lemma}
\begin{proof}
Take any disjoint $M \in {\R^d\choose k + 1}$, $N \in {\R^d \choose d + 1 - k}$ such that the points from $M \cup N$ are in general position.
We have $$d = \dim{\langle M \cup N \rangle} = \dim[\langle M \rangle + \langle N \rangle] = \dim \langle M \rangle + \dim \langle N \rangle - \dim [\langle M \rangle \cap \langle N \rangle] = d - \dim[\langle M \rangle \cap \langle N \rangle].$$ We obtain $\dim[\langle M \rangle \cap \langle N \rangle] = 0$ which means that $|\conv{M} \cap \conv{N}| \leq 1$.
\end{proof}

\begin{remark*}
    By Lemma~\ref{corr} it follows that for any $2k + 3$ points in general position in $\R^{2k}$ the number of intersection points of any $k$-simplex formed by $k + 1$ of the given points with the $k$-faces of the simplex formed by the remaining $k + 2$ points is finite (since the $k$-simplex intersects any $k$-face in at most one point). Hence any $k$-simplex intersects the boundary of the “complementary” $(k + 1)$-simplex at a finite number of points.
\end{remark*}

\begin{proof}[Proof of Remark~\ref{a:extra} modulo Lemma~\ref{corr}]
Denote the given $(2k + 3)$-element set of points by $M$. By Lemma~\ref{corr} the required number equals
\[
    \sum_{\substack{I \in {M \choose k + 1}}} |\conv{I} \cap \partial{\conv{(M \setminus I)}}| = \sum_{\substack{I \in {M \choose k + 1}}}\sum_{\substack{J \in {M \setminus I \choose k + 1}}} |\conv{I} \cap \conv{J}| =
    \sum_{\substack{(I, J) \in {M \choose k + 1}^2 \\ I \cap J = \varnothing}} |\conv{I} \cap \conv{J}| = 0~(mod~2).
\]
The last equality holds since for any unordered pair of disjoint $(k + 1)$-element subsets $I, J \subset M$ there are exactly two terms $|\conv{I} \cap \conv{J}|$ in the sum.
\end{proof}
Now we turn to the proof of Assertion~\ref{a:main}.
\begin{lemma}\label{l:mcl}
    Any $(d - 1)$-dimensional hyperplane intersects the image of the moment curve in $\R^d$ in at most $d$ points.
\end{lemma}

\begin{proof}[Proof]
Take any $(d-1)$-dimensional hyperplane $\pi$ given by the equation $n_1 x_1 + \ldots + n_d x_d - m = 0$. 
A point $(x, x^2, \ldots, x^d)$ lies in $\pi$ if and only if $n_1 x + n_2 x^2 + \ldots + n_d x^d - m = 0$.
This equation has at most $d$ roots.
Hence $\pi$ intersects the image of the moment curve in at most $d$ points.
\end{proof}

Two $p$-element subsets $\{s_1, s_2, \ldots, s_p\}, \{t_1, t_2, \ldots, t_p\}\subset\R$ \emph{alternate} if their elements  can be ordered so that either $s_1 < t_1 < s_2 < t_2 < \ldots < s_{p} < t_{p}$ or $t_1 < s_1 < t_2 < s_2 < \ldots < t_{p} < s_{p}$.

\begin{lemma}\label{l:mcsi} 
Suppose $P, Q \in {\R \choose k + 1}$ are disjoint. Then

(a) the simplices $\conv{m(P)}$ and $\conv{m(Q)}$ intersect if and only if $P$ and $Q$ alternate; and

(b) if the simplices $\conv{m(P)}$ and $\conv{m(Q)}$ intersect, they intersect at one point.
\end{lemma}

\begin{lemma}\label{a:comb}
    For any subset $I \in {[2k + 3] \choose k + 1}$ there is an even number of subsets $J \in {[2k + 3] \setminus I \choose k + 1}$ alternating with $I$.
\end{lemma}

\begin{proof}[Proof of Assertion~\ref{a:main} modulo Lemmas~\ref{l:mcsi},~\ref{a:comb}]
Let us prove that the set $m([2k+3])$ is the required one. The points from this set are in general position by Lemma~\ref{l:mcl}.

Take any $I \in {\R^d \choose k + 1}$.
Then the following numbers are equal:

$\bullet$ the number $n_1$ of intersection points of $\conv{m(I)}$ with the boundary of the simplex $\conv{m([2k + 3] \setminus I)}$;

$\bullet$ the number $n_2$ of $k$-faces of $\conv{m([2k + 3] \setminus I)}$ intersecting $\conv{m(I)}$;

$\bullet$ the number $n_3$ of subsets $J \in {[2k + 3] \setminus I \choose k + 1}$ such that $\conv{m(I)}$ and $\conv{m(J)}$ intersect;

$\bullet$ the number $n_4$ of subsets $J \in {[2k + 3] \setminus I \choose k + 1}$ alternating with $I$.

We have $n_1 = n_2$ since any $k$-face of $\conv{m([2k + 3] \setminus I)}$ intersects $\conv{m(I)}$ at at most $1$ point by Lemma~\ref{corr}.
Also $n_2 = n_3$ since any $k$-face of $\conv{m([2k + 3] \setminus I)}$ is a set $\conv{m(J)}$ for some $J \in {[2k + 3] \setminus I \choose k + 1}$ and any set $\conv{m(J)}$ for $J \in {[2k + 3] \setminus I \choose k + 1}$ is a $k$-face of $\conv{m([2k + 3] \setminus I)}$.
By Lemma~\ref{l:mcsi} we obtain $n_3 = n_4$. By Lemma~\ref{a:comb} $n_4$ is even. Hence $n_1$ is also even.
\end{proof}

\begin{proof}[Proof of Lemma~\ref{l:mcsi}.a] 
Take reals $x_1< x_2< \ldots < x_{2k+2}$ such that $P \cup Q = \{x_1, x_2, \ldots, x_{2k+2}\}$.
The \emph{segment} of the moment curve with the endpoints $x_i$ and $x_{i + 1}$ is $m((x_i, x_{i + 1}))$ for $i \in [2k + 2]$. 
The \emph{midpoint} of such segment is  $m(\frac{x_i + x_{i + 1}}{2})$.
We call the points from $P$ black, and the points from $Q$ white.

Suppose $P$ and $Q$ do not alternate.
Denote by $t$ the number of segments of the moment curve with the endpoints of different colors. 
Then $t\le 2k$. 
Construct a hyperplane $\pi$ containing all the $t$ midpoints of these segments and $2d - t$ points $m(y)$ for $y$ much smaller than $x_1$. 
Since every segment of the moment curve intersecting $\pi$ has endpoints of different colors, $\pi$ separates $\conv{m(P)}$ and $\conv{m(Q)}$. 
Then $\conv{m(P)}$ and $\conv{m(Q)}$ do not intersect.

Suppose $P$ and $Q$ alternate. Take any hyperplane $\pi$ such that $\pi \cap m(P \cup Q) = \varnothing$.
By Lemma~\ref{l:mcl} the hyperplane $\pi$ intersects at most $2k$ segments of the moment curve.
Since $P$ and $Q$ alternate, there are $2k + 1$ segments of the moment curve with the endpoints of different colors. Thus the endpoints of some of these segment lie in the same half-space from $\pi$. This implies that $\pi$ does not separate $\conv{m(P)}$ and $\conv{m(Q)}$. Then $\conv{m(P)}$ and $\conv{m(Q)}$ do not have a separating hyperplane, which means that these sets intersect.
\end{proof}

\begin{proof}[Proof of Lemma~\ref{l:mcsi}.b]
Apply Lemma~\ref{corr} to $d = 2k$, $M := m(P)$, $N := m(Q)$.
\end{proof}

\begin{proof}[Proof of Lemma~\ref{a:comb}]
Suppose that $I$ contains adjacent numbers, i.e., there is $t \in [2k + 2]$ such that $\{t, t + 1\} \subset I$.
Then no subset $J \in {[2k + 3] \choose k + 1}$ alternates with $I$.
Thus the required number of subsets is even.
In the sequel we additionally assume that $I$ contains no adjacent numbers.

We consider four cases:
\emph{(i)} $1, 2k + 3 \in I$;
\emph{(ii)} $1 \in I, 2k + 3 \notin I$; 
\emph{(iii)} $1 \notin I, 2k + 3 \in I$;
\emph{(iv)} $1, 2k + 3 \notin I$.\\
\emph{{Case (i):}}
Since both outer numbers are in $I$, no subset $J \in {[2k + 3] \choose k + 1}$ alternates with $I$.
Then the required number of subsets is even. \\
\emph{Case (ii):}
Since $I$ contains no adjacent numbers and only one outer number, the set $[2k + 3] \setminus I$ can be uniquely represented as a union $J_1 \sqcup J_2 \sqcup \ldots \sqcup J_{k + 1}$, where $J_i$ are nonempty sets of consecutive integers or single integers. Since $|J_1| + |J_2| + \ldots + |J_{k + 1}| = 2k + 3 - (k + 1) = k + 2$, among $J_i$ there exists a unique $2$-element set and the others have only one element.
Thus the number of subsets $J \in {[2k + 3] \choose k + 1}$ alternating with $I$ is exactly
$|J_1| \cdot |J_2| \cdot \ldots \cdot |J_{k + 1}| = 1 \cdot \ldots \cdot 1 \cdot 2 \cdot 1 \cdot \ldots \cdot 1 = 2$, which is even.\\
\emph{Case (iii):}
Analogously to Case (ii).\\
\emph{Case (iv):}
Since $I$ contains neither adjacent numbers nor outer ones, $I = \{2, 4, \ldots, 2k + 2\}$. Then the only $(k + 1)$-element subsets of $[2k + 3]$ alternating with $I$ are $\{1, 3, 5, \ldots, 2k + 1\}$ and $\{3, 5, 7, \ldots, 2k + 3\}$.
Thus the number of subsets $J \in {[2k + 3] \choose k + 1}$ alternating with $I$ equals $2$, which is even.
\end{proof}


\begin{thebibliography}{RSS95}
\bibitem[Br]{Br} \emph{M. Breen.} Primitive Radon partitions for cyclic polytopes, Israel Journal of Mathematics, 15 (1973), 156-157. 

\bibitem[Sk]{Sk22} \emph{A. Skopenkov.} Realizability of hypergraphs and intrinsic link theory, Mat. Prosveschenie, 32 (2023), 125-159, arXiv:1402.0658.
\end{thebibliography}
\end{document}